\documentclass[11pt, a4paper]{amsart}

\usepackage[usenames]{color}
\usepackage{amssymb , array, amsmath, amscd, bbm, epsfig, graphicx, pdfpages, eurosym, multirow, multicol, wrapfig, enumerate, amsthm, hyperref, fixltx2e, setspace}

\DeclareMathOperator*{\Ran}{Ran}
\DeclareMathOperator*{\Ker}{Ker}

\DeclareMathOperator*{\R}{Re}

\DeclareMathOperator*{\Sp}{Sp}
\DeclareMathOperator*{\Spu}{Sp_\mathrm{u}}

\DeclareMathOperator*{\Spuw}{Sp_u^\textit{w}}

\newcommand{\ii}{\mathrm{i}}
\newcommand{\dd}{\mathrm{d}}

\newtheorem{thm}{Theorem}[section]

\theoremstyle{definition}

\newtheorem{rem}[thm]{Remark}
\newtheoremstyle{mystyle}{}{}{\sl}{}{\bf}{.}{ }{}
\theoremstyle{mystyle}
\newtheorem{mythm}[thm]{Theorem}
\newtheorem{mylem}[thm]{Lemma}
\newtheorem{myprp}[thm]{Proposition}
\newtheorem{mycor}[thm]{Corollary}

\numberwithin{equation}{section}  

\begin{document}

\title[Improvements of the Katznelson-Tzafriri theorem]{Some improvements of the Katznelson-\\Tzafriri theorem on Hilbert space}
\author{David Seifert}
\address{Mathematical Institute, 24--29 St Giles', Oxford\;\;OX1 3LB, United Kingdom}
\curraddr{Balliol College, Oxford\;\;OX1 3BJ, United Kingdom}
\email{david.seifert@balliol.ox.ac.uk}
\date{27 March 2013}

\begin{abstract}
 This paper extends two recent improvements in the Hilbert space setting of the well-known Katznelson-Tzafriri theorem  by establishing both a version of the result valid for bounded representations of a  large class of abelian semigroups and a quantified version for contractive representations. The paper concludes with an outline of an improved version of the Katznelson-Tzafriri theorem for individual orbits, whose validity extends even to certain unbounded representations.
\end{abstract}

\maketitle

\section{Introduction}

In \cite{KT}, Katznelson and Tzafriri proved that, given a power-bounded operator $T$ on a complex Banach space $X$, $\|T^n(I-T)\|\rightarrow 0$ as $n\rightarrow\infty$ if (and only if) $\sigma(T)\cap\mathbb{T}\subset\{1\}$, where $\sigma(T)$ denotes the spectrum of $T$ and $\mathbb{T}$ is the unit circle. Given a sequence $a\in\ell^1(\mathbb{Z}_+)$, define $\widehat{a}(\lambda):=\sum_{n=0}^\infty a(n)\lambda^n$ ($|\lambda|\leq1$) and $\widehat{a}(T):=\sum_{n=0}^\infty a(n)T^n$, which is a bounded linear operator on $X$. Katznelson and Tzafriri also showed, in the same paper, that \begin{equation}\label{discrete_KT}\lim_{n\rightarrow\infty}\|T^n\widehat{a}(T)\|=0\end{equation} provided there exists a sequence $(a_n)$ in $\ell^1(\mathbb{Z}_+)$ such that each $\widehat{a_n}$ vanishes on an open neighbourhood of $\sigma(T)\cap\mathbb{T}$ and $\|a-a_n\|_1\rightarrow0$ as $n\rightarrow\infty$. This result has itself subsequently  been extended, first to the case of $C_0$-semigroups (see \cite[Th\'eor\`eme~III.4]{ESZ} and \cite[Theorem~3.2]{Vu}) and later to more general semigroup representations (see \cite[Theorem~4.3]{BV} and \cite{Vu}).

These results are optimal in various senses (see \cite[Section~5]{CT}), but improvements are possible when $X$ is assumed to be a Hilbert space. It is shown in  \cite[Corollary~2.12]{ESZ2}, for instance, that the weaker (and necessary) condition that $\widehat{a}$ vanish on $\sigma(T)\cap\mathbb{T}$ is sufficient for \eqref{discrete_KT} to hold, at least when $T$ is a contraction; see \cite[Proposition~1.6]{KvN} for a slightly more general result. This result has in turn been improved upon in two recent papers. In \cite{Leka}, L\'eka has extended this result to power-bounded operators on Hilbert space, and Zarrabi in \cite{Zarr} has shown that for contractions, and likewise for pairs of commuting contractions and for contractive $C_0$-semigroups, the limit appearing in \eqref{discrete_KT} is given, more generally, by  $\sup\{ |\widehat{a}(\lambda)|:\lambda\in\sigma(T)\cap\mathbb{T}\}$ or the appropriate analogue; for related results see \cite{AR},  \cite{BBG}, \cite{Berc2}, \cite{Berc} and \cite{Mus}.

The purpose of this paper is to improve both L\'eka's and Zarrabi's versions of the Katznelson-Tzafriri theorem by extending them to representations of a significantly larger class of semigroups. The main result, Theorem~\ref{Thm},  is a  Katznelson-Tzafriri type theorem which holds for bounded (as opposed to contractive) representations and thus includes  both \cite[Theorem~2.1]{Leka} and various results contained as special cases in \cite{Zarr}. The main implication of  Theorem~\ref{Thm} is proved first  via a certain ergodic condition as in \cite{Leka} and then by a more direct argument. The second method does not involve the ergodic condition and, as is shown in Theorem~\ref{ub}, leads naturally to an extension of Zarrabi's quantified results for contractive representations. Section~5, finally, contains a brief exposition of an improved version of the Katznelson-Tzafriri theorem for individual orbits. First of all, though, Section~2 sets out the necessary preliminary material.

\section{Preliminaries}

The setting throughout, even when not stated explicitly, will be that of an abelian semigroup $S$ contained in a locally compact abelian group $G$ satisfying $G=S-S$. The Haar measure on $G$ is denoted by $\mu$, and it is assumed that $S$ is Haar-measurable and hence itself becomes a measure space with respect to the restriction of $\mu$. Assume furthermore that the interior $S^\circ$ of $S$ (in the topology induced by $G$) is non-empty. The semigroup $S$ becomes a directed set under the relation $\succeq$, where $s\succeq t$ for $s,t\in S$ whenever $s-t\in S\cup\{0\}$; this makes it possible to speak of limits as $s\rightarrow\infty$ through $S$.  The dual group of $G$, consisting of all continuous bounded characters  $\chi:G\rightarrow\mathbb{C}$, is denoted by $\Gamma$, the set of continuous bounded characters on $S$ by $S^*$. It follows from the assumption that $S$ spans $G$ that the subset of $S^*$ of characters taking values in the unit circle $\mathbb{T}$ can (and will throughout) be identified with $\Gamma$. Two important examples of the above are the (semi)groups $\mathbb{Z}_{(+)}$ with counting measure and $\mathbb{R}_{(+)}$ with  Lebesgue measure. Here $S^*$ can be identified in a natural way with $\{\lambda\in\mathbb{C}:|\lambda|\leq1\}$ and $\{\lambda\in\mathbb{C}:\R \lambda\leq0\}$, respectively, and the dual group $\Gamma$ is $\mathbb{T}$ and $\ii\mathbb{R}$ in each case.

For $\Omega=G$ or $S,$ let  $L^1(\Omega)$ denote the algebra (under convolution) of functions $a:\Omega\rightarrow\mathbb{C}$ that are integrable with respect to (the restriction of) Haar measure and, given $a\in L^1(\Omega)$, define its Fourier transform  by $$\widehat{a}(\chi):=\int_\Omega a(s)\chi(s)\,\mathrm{d}\mu(s),$$ where $\chi$ is an element of $\Gamma$ or $S^*$, as appropriate. Given a closed subset $\Lambda$ of $\Gamma$, define  $J_\Lambda:=\{a\in L^1(G): \widehat{a}\equiv0\;\mbox{in a neighbourhood of $\chi\in\Lambda$}\}$ and $K_\Lambda:=\{a\in L^1(G): \widehat{a}(\chi)=0\;\mbox{for all $\chi\in\Lambda$}\}$. An element of $L^1(G)$ is said to be of \textsl{spectral synthesis} with respect to $\Lambda$ if it lies in the closure of $J_\Lambda$. Since $K_\Lambda$ is closed, any such function must be an element of $K_\Lambda$. If $K_\Lambda$ coincides with the closure of $J_\Lambda$, the set $\Lambda$ is said to be of \textsl{spectral synthesis}. 

For any closed subset $\Lambda$ of $\Gamma$, the map $W_\Lambda: a\mapsto \widehat{a}|_\Lambda$ is a well-defined contractive algebra homomorphism from $L^1(G)$ into $ C_0(\Lambda)$  whose kernel is $K_\Lambda$ and whose range, by the Stone-Weierstrass theorem, is dense in $C_0(\Lambda)$; see \cite[Theorem~1.2.4]{Rudin}. Since $K_\Lambda(G)$ is a closed ideal of $L^1(G)$,  $W_\Lambda$ induces a well-defined injective algebra homomorphism $ U_\Lambda:L^1(G)/K_\Lambda(G)\rightarrow C_0(\Lambda)$ which, by the Inverse Mapping Theorem, is an isomorphism precisely when it  is surjective. Since its range is dense in $C_0(\Lambda)$, this is the case if and only if the map is an isomorphic embedding (which in turn is equivalent to the dual operator $ U_\Lambda':M(\Lambda)\rightarrow K_\Lambda^\perp$ being either a surjection or an isomorphic embedding, where $M(\Lambda)$ denotes the set of  complex-valued regular measures on $\Lambda$ which have finite total variation; see for instance \cite[Appendix~C11]{Rudin}). When these conditions are satisfied, $\Lambda$ is said to be a \textsl{Helson set}, and the quantity $\alpha(\Lambda):=\| U_\Lambda^{-1}\|$ is known as its \textsl{Helson constant}. Since $U_\Lambda$ is contractive, $\alpha(\Lambda)\geq1$ for any Helson set $\Lambda$. For some examples of Helson sets, see \cite[Remark~5.6]{Zarr}.

In what follows, it will be assumed that the set $\{\widehat{a}:a\in L^1(S)\}$  separates points both from each other and from zero and, furthermore, that the interior $S^\circ$ is dense in $S$. For further details and discussion of these conditions, see for instance \cite{BV}. These assumptions ensure, in particular, that there exists a net $(\Omega_\alpha)$, known as a \textsl{F\o lner net}, of compact, measurable, non-null subsets of $S$ satisfying $$\lim_{\alpha\rightarrow\infty}\frac{\mu\big(\Omega_\alpha\bigtriangleup(\Omega_\alpha+s)\big)}{\mu(\Omega_\alpha)}=0,$$uniformly for $s$ in compact subsets of $S$. 

Given a Banach space $X$ and $\Omega=G$ or $S$ for $S$ and $G$ as above, a \textsl{representation} of  $\Omega$ on $X$ is a strongly continuous homomorphism $T:\Omega\rightarrow\mathcal{B}(X)$ which, if $0\in \Omega$, satisfies $T(0)=I.$ The representation is said to be \textsl{bounded} if $\sup\{\|T(s)\|:s\in\Omega\}<\infty$ and in this case, given $a\in L^1(\Omega)$, the operator $\widehat{a}(T)\in\mathcal{B}(X)$ is defined, for each $x\in X$, by $$\widehat{a}(T)x:=\int_\Omega a(s)T(s)x\,\dd\mu(s).$$ 

Given a bounded representation $T$ of $G$ on a Banach space $X$ and a closed subspace $\Lambda$ of the dual group $\Gamma$, the corresponding \textsl{spectral subspace} $M_T(\Lambda)$ is defined as  $$M_T(\Lambda):=\bigcap\nolimits_{a\in J_\Lambda}\Ker\widehat{a}(T),$$  the \textsl{(Arveson) spectrum} $\Sp(T)$ of $T$ as $$\Sp(T):=\bigcap\big\{\Lambda\subset\Gamma:\mbox{$\Lambda$ is closed and $M_T(\Lambda)=X$}\big\}.$$ Thus the spectrum is a closed subset of $\Gamma$, and it is shown in \cite[Theorems~8.1.4 and 8.1.12]{Ped}, respectively, that  $\Sp(T)$ is non-empty whenever $X$ is non-trivial and that it is compact if and only if $T$ is continuous with respect to the norm topology on $\mathcal{B}(X)$. If $T$ is a representation by isometries, this notion of spectrum coincides with the \textsl{finite L-spectrum} of \cite[Section~5.2]{Lyu}. By \cite[Proposition~8.1.9]{Ped}, furthermore, $\Sp(T)$ has the alternative description $$\Sp(T)=\big\{\chi\in\Gamma:|\widehat{a}(\chi)|\leq\|\widehat{a}(T)\|\;\mbox{for all $a\in L^1(G)$}\big\}.$$ Accordingly, given a bounded representation $T$ of a semigroup  $S$ on a Banach space $X$, the \textsl{spectrum} $\Sp(T)$ of $T$ is defined as $$\Sp(T):=\big\{\chi\in S^*:|\widehat{a}(\chi)|\leq\|\widehat{a}(T)\|\;\mbox{for all $a\in L^1(S)$}\big\},$$ and the \textsl{unitary spectrum} of $T$ is given by $\Spu(T):=\Sp(T)\cap\Gamma$; see \cite{BV} for  details. In the examples mentioned above, bounded semigroup representations correspond to a single power-bounded operator $T\in\mathcal{B}(X)$  if $S=\mathbb{Z}_+$ and to a bounded $C_0$-semigroup if $S=\mathbb{R}_+$. The spectrum is given by $\sigma(T)$ and $\sigma(A)$, respectively, where $A$ denotes the generator of the semigroup.

\section{A general Katznelson-Tzafriri type result}

The aim of this section is to prove the following generalisation of \cite[Theorem~2.1]{Leka}.

\begin{mythm}\label{Thm}
Let $T$ be a bounded representation of a semigroup $S$ on a Hilbert space $X$, and suppose that $a\in L^1(S)$. Then the following are equivalent:
\begin{enumerate}
\item[(i)]$\widehat{a}(\chi)=0$ for every $\chi\in\Spu(T);$
\item[(ii)] Given any F\o lner net $(\Omega_\alpha)$ for $S$ and any $\chi\in \Spu(T)$, \begin{equation}\label{erg}\lim_{\alpha\rightarrow\infty}\frac{1}{\mu(\Omega_\alpha)}\left\|\int_{\Omega_\alpha}\overline{\chi(s)}T(s)\widehat{a}(T)\,\mathrm{d}\mu(s)\right\|=0;\end{equation}
\item[(iii)] $\|T(s)\widehat{a}(T)\|\rightarrow0$ as $s\rightarrow\infty$.
\end{enumerate}
\end{mythm}

\begin{rem}\label{Q_rem}
In \cite[Theorem~2.1]{Leka}, conditions (ii) and (iii) above are presented in a slightly more general form, with the operator $\widehat{a}(T)$ replaced by an arbitrary  $Q\in\mathcal{B}(X)$ that commutes with the representation. The presentation here is restricted to the case $Q=\widehat{a}(T)$ purely for simplicity.
\end{rem}

The proof of this result will be broken up into a number of separate steps, all of which correspond to some part of the proof of \cite[Theorem~2.1]{Leka} given by L\'eka but typically with some modifications to accommodate the more general setting in which the representation need not be norm continuous. The following lemma constitutes the main step towards proving that (i)$\implies$(ii); it corresponds to \cite[Lemma~2.2]{Leka}. Note that the Hilbert space assumption is not required for this part of the argument.

\begin{mylem}\label{int_lem} Let $T$ be a bounded representation of a semigroup $S$ on a Banach space $X$, and let $a\in L^1(S)$.  Then, for all $\chi\in\Gamma$,  $$\lim_{\alpha\rightarrow\infty}\frac{1}{\mu(\Omega_\alpha)}\left\|\int_{\Omega_\alpha}\overline{\chi(s)}T(s)\big(\widehat{a}(T)-\widehat{a}(\chi)\big)\,\mathrm{d}\mu(s)\right\|=0,$$ where $(\Omega_\alpha)$ is any F\o lner net  for $S$ and the integral is taken in the strong sense. 
\end{mylem}

\begin{proof}[\textsc{Proof}]
With $a\in L^1(S)$ and $\chi\in\Gamma$ fixed, let $x\in X$ have unit norm and set $$I_x(\alpha):=\frac{1}{\mu(\Omega_\alpha)}\left\|\int_{\Omega_\alpha}\overline{\chi(s)}T(s)\big(\widehat{a}(T)x-\widehat{a}(\chi)x\big)\,\mathrm{d}\mu(s)\right\|.$$ Then, by a simple application of Fubini's theorem, $$\begin{aligned}I_x(\alpha)&\leq M\int_S\frac{\mu\big(\Omega_\alpha\bigtriangleup(\Omega_\alpha+s)\big)}{\mu(\Omega_\alpha)}|a(s)|\,\mathrm{d}\mu(s),\end{aligned}$$ where $M:=\sup\{\|T(s)\|:s\in S\}$. Let $\varepsilon>0$. Since $a\in L^1(S)$, there exists a compact subset $K$ of $S$ such that $\int_{S\backslash K}|a(s)|\,\mathrm{d}\mu(s)<\varepsilon/4M.$ Defining $$\xi_K(\alpha):=\sup\left\{\frac{\mu(\Omega_\alpha\bigtriangleup(\Omega_\alpha+s))}{\mu(\Omega_\alpha)}:s\in K\right\},$$ it follows from the definition of a F\o lner net that $\xi_K(\alpha)\rightarrow0$ as $\alpha\rightarrow\infty$. Since  $$I_x(\alpha)\leq  M\|a\|_1\xi_K(\alpha)+2M\int_{S\backslash K}|a(s)|\,\mathrm{d}\mu(s),$$ $I_x(\alpha)<\varepsilon$ for all sufficiently large $\alpha$, and the result follows.
\end{proof}

\begin{mycor}\label{cor:erg_cond}
Let $T$ be a bounded representation of a semigroup $S$ on a Banach space $X$, and let $\chi\in\Gamma$. Suppose that $a\in L^1(S)$ is such that $\widehat{a}(\chi)=0$. Then \eqref{erg} holds for any  F\o lner net  $(\Omega_\alpha)$ for $S$. 
\end{mycor}

The next result is an important step towards establishing the implication (ii)$\implies$(iii) in Theorem~\ref{Thm} and should be compared with \cite[Lemma~2.4]{Leka}.

\begin{myprp}\label{zero}
Let $S$ be a semigroup and  $T$ a representation of a group $G=S-S$ by unitary operators on a Hilbert space $X$. Suppose that $a\in L^1(G)$ and that, for each $\chi\in \Sp(T)$, $$\lim_{\alpha\rightarrow\infty}\frac{1}{\mu(\Omega_\alpha)}\left\|\int_{\Omega_\alpha}\overline{\chi(s)}T(s)\widehat{a}(T)\,\mathrm{d}\mu(s)\right\|=0,$$ where $(\Omega_\alpha)$ is any F\o lner net for $S$. Then $\widehat{a}(T)=0$.
\end{myprp}

\begin{proof}[\textsc{Proof}]
Writing $B(\Gamma)$ for the set of Borel subsets of the dual group $\Gamma$, let $E:B(\Gamma)\rightarrow \mathcal{B}(X)$ denote the spectral measure associated with $T$ (see \cite[Theorem~8.3.2]{Ped}) and, for $s\in G$ and $\Lambda\in B(\Gamma)$, let $T_\Lambda(s):=T(s)E(\Lambda)$. Then \begin{equation*}\label{spectral1}T_\Lambda(s):=\int_\Lambda \chi(s)\,\mathrm{d}E(\chi),\end{equation*} the integral being taken in the weak sense, and, by Fubini's theorem,  \begin{equation}\label{spectral}\widehat{b}(T_\Lambda)=\int_\Lambda \widehat{b}(\chi)\,\mathrm{d}E(\chi)\end{equation} for all $b\in L^1(G)$. Thus if $\Lambda\in B(\Gamma)$ is closed and $b\in J_\Lambda(G)$, then $\widehat{b}(T_\Lambda)=0$, and it follows that $M_{T_\Lambda}(\Lambda)=X$, so that $\Sp(T_\Lambda)\subset\Lambda$. Choosing $\Lambda\in B(\Gamma)$ to be compact ensures that the representation $T_\Lambda$ of $G$ on $X$ is norm continuous.

Set $Q:=\widehat{a}(T)$ and, for  a given compact subset $\Lambda$ of $\Sp(T)$, define  $Q_\Lambda:=QE(\Lambda)$, noting that $Q_\Lambda$ is normal and that $Q_\Lambda\rightarrow Q$ in the weak (and indeed the strong)  operator topology as $\Lambda$ approaches $\Sp(T)$ through compact subsets. Furthermore, let $\mathcal{A}_\Lambda$ denote the commutative unital $C^*$-algebra  generated by $\{Q_\Lambda,Q_\Lambda^*\}\cup\{T_\Lambda(s):s\in G\}$, and  let  $\Delta(\mathcal{A}_\Lambda)$ denote its character space. Write $\Phi_\Lambda:\mathcal{A}_\Lambda\rightarrow C_0(\Delta(\mathcal{A}_\Lambda))$ for the Gelfand transform of $\mathcal{A}_\Lambda$, which is an isometric $*$-isomorphism, and consider the map $\chi_\xi:G\rightarrow\mathbb{C}\backslash\{0\}$ given, for $\xi\in\Delta(\mathcal{A}_\Lambda)$ and $s\in G$, by $\chi_\xi(s):= \Phi_\Lambda(T_\Lambda(s))(\xi).$ Since the representation $T_\Lambda$ is norm continuous, $\chi_\xi$ is a continuous group homomorphism, and the fact that each $\xi\in\Delta(\mathcal{A}_\Lambda)$ is a bounded linear functional on $\mathcal{A}_\Lambda$ with $\|\xi\|=|\xi(E(\Lambda))|=1$ implies that $|\chi_\xi(s)|\leq1$ for all $s\in G$. Hence $\chi_\xi\in\Gamma$. Moreover, if $b\in L^1(G)$, then $$\big|\widehat{b}(\chi_\xi)\big|=\left|\xi\left(\int_Gb(s)T_\Lambda(s)\,\mathrm{d}\mu(s)\right)\right|\leq\big\|\widehat{b}(T_\Lambda)\big\|,$$ which is to say that $\chi_\xi\in \Sp(T_\Lambda)$, and hence $\chi_\xi\in\Sp(T)$. Let $g_\Lambda:=\Phi_\Lambda(Q_\Lambda)$. Then $$\begin{aligned}|g_\Lambda(\xi)|&=\frac{1}{\mu(\Omega_\alpha)}\left|\int_{\Omega_\alpha}|\chi_\xi(s)|^2g_\Lambda(\xi)\,\mathrm{d}\mu(s)\right|\\&\leq \frac{1}{\mu(\Omega_\alpha)}\left\|\Phi_\Lambda\left(\int_{\Omega_\alpha}\overline{\chi_\xi(s)}T_\Lambda(s)Q_\Lambda\,\mathrm{d}\mu(s)\right)\right\|_\infty\\&=\frac{1}{\mu(\Omega_\alpha)}\left\|\int_{\Omega_\alpha}\overline{\chi_\xi(s)}T_\Lambda(s)Q_\Lambda\,\mathrm{d}\mu(s)\right\|\\&\leq \frac{1}{\mu(\Omega_\alpha)}\left\|\int_{\Omega_\alpha}\overline{\chi_\xi(s)}T(s)Q\,\mathrm{d}\mu(s)\right\|,\end{aligned}$$for any $\xi\in \Delta(\mathcal{A}_\Lambda)$ and letting $\alpha\rightarrow\infty$ shows that $g_\Lambda=0.$ Since $\Phi_\Lambda$ is an isometry, it follows that $Q_\Lambda=0$, and allowing $\Lambda$ to approach $\Sp(T)$ through compact subsets gives $Q=0$, as required. 
\end{proof}

\begin{rem}
The result remains true when $\widehat{a}(T)$ is replaced by any $Q\in\mathcal{B}(X)$ which commutes with $T$. If $Q$ is normal, this follows from the same argument as above; the general case can then be obtained by considering the operator $Q^* Q$; see also \cite[Lemma~2.4]{Leka}.
\end{rem}

 Propositions~\ref{pwlim} and \ref{lim} below correspond in essence to the two main stages in the proof of \cite[Theorem~2.1]{Leka} and show,  via an intermediate result for the strong operator topology,  that  (ii)$\implies$(iii) in Theorem~\ref{Thm}.

\begin{myprp}\label{pwlim}
 Let $T$ be a bounded representation of a semigroup $S$ on a Hilbert space $X$, and let $a\in L^1(S)$. Suppose that, for some F\o lner net $(\Omega_\alpha)$  for $S$, \eqref{erg} holds for all $\chi\in\Spu(T)$. Then, for all $x\in X$, $\|T(s)\widehat{a}(T)x\|\rightarrow0$ as $s\rightarrow\infty$.
\end{myprp}

\begin{proof}[\textsc{Proof}]
Fix a Banach limit $\phi$ on $L^\infty(S)$. A construction analogous to \cite[Proposition~3.1]{BV} and  \cite[Section~1]{KvN} shows that there exist a Hilbert space $X_\phi$, a representation $T_\phi$ of $S$ on $X_\phi$ by isometries with $\Sp(T_\phi)\subset\Sp(T)$ and an operator $\pi_\phi:X\rightarrow X_\phi$ with the following properties: $\pi_\phi$ is bounded with norm $\|\pi_\phi\|\leq \sup\{\|T(s)\|:s\in S\}$, $\|\pi_\phi(x)\|^2=\phi(\|T(\cdot)x\|^2)$ for all $x\in X$,  $\Ran\pi_\phi$ is dense in $X_\phi$, $\Ker\pi_\phi=\{x\in X:\|T(s)x\|\rightarrow0\;\mbox{as}\;s\rightarrow\infty\}$
 and $\pi_\phi T(s)= T_\phi(s)\pi_\phi$ for all $s\in S$ (so $\pi_\phi$ is an \textsl{intertwining operator}). In particular, for any  operator $Q\in\mathcal{B}(X)$ that commutes with $T$, the operator  $Q_\phi\in\mathcal{B}(X_\phi)$ defined by $\pi_\phi Q=Q_\phi\pi_\phi$ satisfes $\|Q_\phi\|\leq\|Q\|$. By a construction analogous to  \cite[Proposition~2.1]{BG} (see also \cite[Proposition~3.2]{BV}, \cite{BY}, \cite{Dou} and \cite{Ito}), there exist a further Hilbert space $Y_\phi$, a representation $T_G$ of the group $G=S-S$ by unitary operators on $Y_\phi$ with $\Sp(T_G)=\Spu(T_\phi)$ and an isometric intertwining operator $\pi_G:X_\phi\rightarrow Y_\phi$ such that $\{T_G(-s)\pi_G(x):s\in S, x\in X_\phi\}$ is dense in $Y_\phi$. The latter implies, in particular, that $\|Q_G\|=\|Q_\phi\|$ for all $Q_\phi\in\mathcal{B}(X_\phi)$ and all $Q_G\in\mathcal{B}(X_G)$ which commute with $T_G$ and satisfy $\pi_G Q_\phi=Q_G\pi_G$. Thus it is possible to assume, sacrificing only the density condition on the range of the intertwining operator, that $T_\phi$ itself is in fact a representation of $G$ by unitary operators on $X_\phi$. 

Now, given $\chi\in\Sp(T_\phi)$, define operators $Q_\alpha\in\mathcal{B}(X)$ and $Q_{\phi,\alpha}\in\mathcal{B}(X_\phi)$ as \begin{equation}\label{Ba}Q_\alpha:=\frac{1}{\mu(\Omega_\alpha)}\int_{\Omega_\alpha}\overline{\chi(s)}T(s)\widehat{a}(T)\,\mathrm{d}\mu(s)\end{equation} and $$Q_{\phi,\alpha}:=\frac{1}{\mu(\Omega_\alpha)} \int_{\Omega_\alpha}\overline{\chi(s)}T_\phi(s)\widehat{a}(T_\phi)\,\mathrm{d}\mu(s).$$   Then $\pi_\phi Q_\alpha=Q_{\phi,\alpha}\pi_\phi$, from which it follows  that $\|Q_{\phi,\alpha}\|\leq\|Q_\alpha\|$. In particular, $\|Q_{\phi,\alpha}\|\rightarrow0$ as $\alpha\rightarrow\infty$. Identifying $L^1(S)$  in the natural way with a subset of $L^1(G)$, it follows from  Proposition~\ref{zero} applied to $T_\phi$, $X_\phi$ and $a$ that $\widehat{a}(T_\phi)=0$. In particular, $\pi_\phi(\widehat{a}(T)x)=\widehat{a}(T_\phi)\pi_\phi(x)=0$ for any $x\in X$, so the result follows from the description of $\Ker\pi_\phi$.\end{proof}

\begin{myprp}\label{lim}
 Let $T$ be a bounded representation of a semigroup $S$ on a Hilbert space $X$, and let $a\in L^1(S)$. Suppose that, for some F\o lner net $(\Omega_\alpha)$  for $S$, equation~\eqref{erg} is satisfied for all $\chi\in\Spu(T)$. Then $\|T(s)\widehat{a}(T)\|\rightarrow0$ as $s\rightarrow\infty.$
\end{myprp}

\begin{proof}[\textsc{Proof}]
Suppose, for the sake of contradiction, that \eqref{erg} holds for all $\chi\in\Spu(T)$ and some F\o lner net $(\Omega_\alpha)$ but that there exist  $\varepsilon>0$ and a net $(s_\beta)$ in $S$, with indexing set $A$ say, such that $s_\beta\rightarrow\infty$ as $\beta\rightarrow\infty$ and, for some suitable sequence $(y_\beta)$ of unit vectors in $X$,   $\|T(s_\beta)\widehat{a}(T)y_\beta\|\geq \varepsilon$ for all $\beta\in A$.  Letting $M:=\sup\{\|T(s)\|:s\in S\}$,  it follows that $\|T(s)\widehat{a}(T)y_\beta\|\geq \varepsilon M^{-1}$ whenever $s_\beta-s\in S$. Fix  $t\in S^\circ$ and, essentially as in \cite[Section~1.1.8]{Rudin}, let  $b\in L^1(S)$ satisfy $\|b\|_1=1$ and  $\|a*b-a_{t}\|_1<\varepsilon/2M^3$, where $a_{t}\in L^1(S)$ is given by $a_{t}(s)=a(s-t)$ if $s-t\in S$ and $a_{t}(s)=0$ otherwise. Now, by a construction similar those contained in  \cite{HR}, \cite{NR} and \cite{RW}, there exist 
\begin{enumerate}
\item[(a)] a Banach space $X_A^\infty$, which is contained in the set $\ell^\infty(A;X)$ of $X$-valued nets indexed by $A$ and  contains all nets of the form $(\widehat{c}(T)x_\alpha)$ with $c\in L^1(S)$ and $(x_\alpha)\in \ell^\infty(A;X)$, and a bounded representation $T_A^\infty$ of $S$ on $X_A^\infty$ with $\Sp(T_A^\infty)=\Sp(T)$;

\item[(b)]a Hilbert space $X_A$, a bounded representation $T_A$ of $S$ on $X_A$ with $\Sp(T_A)\subset \Sp(T_A^\infty)$ and a surjective intertwining operator $\pi_A:X_A^\infty\rightarrow X_A$ which is contractive and such that $\|\pi_A(x_\alpha)\|$ is given, for each $(x_\alpha)\in X_A^\infty$, by the limit of the net $(\|x_\alpha\|)$ along some ultrafilter on $A$ which contains the filter generated by the sets $\{\alpha\in A:\alpha\succeq\beta\}$ with $\beta\in A$.
\end{enumerate}

Note, in particular, that $\Sp(T_A)\subset \Sp(T)$. Consider the element $(x_\beta)$ of $X_A^\infty$, where $x_\beta:=\widehat{b}(T)y_\beta$. Then, writing $c:=a*b-a_{t}$, \begin{equation*}\label{long}\begin{aligned}\|T_A(s)\widehat{a}(T_A)\pi_A(x_\beta)\|&=\|\pi_AT_A^\infty(s)\widehat{a}(T_A^\infty)(x_\beta)\|\\&=\big\|\pi_A\big(T(s)\widehat{a*b}(T)y_\beta\big)\big\|\\&\geq\big\|\pi_A\big(T(s+t)\widehat{a}(T)y_\beta\big)\big\|-\big\|\pi_A\big(T(s)\widehat{c}\,(T)y_\beta\big)\big\|\\&\geq\liminf_{\beta\rightarrow\infty}\|T(s+t)\widehat{a}(T)y_\beta\|-M^2\|c\|_1
\end{aligned}\end{equation*}  for all $s\in S$, where the last line follows from the definition of the norm on $X_A$. Thus $\|T_A(s)\widehat{a}(T_A)\pi_A(x_\beta)\|\geq \varepsilon/2M$ for all $s\in S$.

Fix $\chi\in\Spu(T_A)$ and define the operators $Q_{A,\alpha}^\infty\in\mathcal{B}(X_A^\infty)$ and $Q_{A,\alpha}\in\mathcal{B}(X_A)$ as $$Q_{A,\alpha}^\infty:=\frac{1}{\mu(\Omega_\alpha)}\int_{\Omega_\alpha}\overline{\chi(s)}T_A^\infty(s)\widehat{a}(T_A^\infty)\,\mathrm{d}\mu(s)$$ and $$Q_{A,\alpha}:=\frac{1}{\mu(\Omega_\alpha)}\int_{\Omega_\alpha}\overline{\chi(s)}T_A(s)\widehat{a}(T_A)\,\mathrm{d}\mu(s).$$   Then $\pi_A Q_{A,\alpha}^\infty=Q_{A,\alpha}\pi_A$, so  the properties of $\pi_A$ and the fact that $Q_{A,\alpha}^\infty$ acts on $X_A^\infty$ by entrywise application of the operator $Q_\alpha$, as defined in equation~\eqref{Ba}, imply that $\|Q_{A,\alpha}\|\leq\|Q_\alpha\|$. Hence $\|Q_{A,\alpha}\|\rightarrow0$ as $\alpha\rightarrow\infty$, and Proposition~\ref{pwlim} applied to $T_A$ and $X_A$ leads to the required contradiction.
\end{proof}

Corollary~\ref{cor:erg_cond} and Proposition~\ref{lim} together prove the implications (i)$\implies$(ii)$\implies$(iii) of Theorem~\ref{Thm}. The following simple lemma, which follows immediately from the definition of the spectrum of a semigroup representation $T$ along with the observation that $\widehat{a_s}(T)=T(s)\widehat{a}(T)$ for all $a\in L^1(S)$ and $s\in S$, shows that (iii)$\implies$(i), thus completing the proof of the main result.

\begin{mylem}\label{converse2}
Let $T$ be a bounded representation of a semigroup $S$ on a Banach space $X$, and let $a\in L^1(S)$. Then $|\widehat{a}(\chi)|\leq\|T(s)\widehat{a}(T)\|$ for all $\chi\in\Spu(T)$ and all $s\in S$.
\end{mylem}

\begin{rem}
There is a direct proof of the implication (iii)$\implies$(ii) in Theorem~\ref{Thm}. Indeed, if  $T$ is a bounded representation of a semigroup $S$ on any Banach space $X$, if $a\in L^1(S)$  and if $(\Omega_\alpha)$ is any F\o lner net for $S$, then $$\|Q_\alpha\|\leq \sup\{\|T(s)\widehat{a}(T)\|:s\succeq t\}+ M^2\|a\|_1\frac{\mu\big(\Omega_\alpha\bigtriangleup(\Omega_\alpha+t)\big)}{\mu(\Omega_\alpha)}$$ for any $t\in S$, where $Q_\alpha$ is as in \eqref{Ba} and  $M=\sup\{\|T(s)\|:s\in S\}$. Hence (iii)$\implies$(ii)  by definition of a F\o lner net.  Moreover, it is possible, at least in special cases, to show directly that (ii)$\implies$(i). When $S=\mathbb{Z}_+$, this follows from  Corollary~\ref{cor:erg_cond}  and the uniform ergodic theorem (see \cite[Corollary~2.3]{Leka}), and a similar argument works when $S=\mathbb{R}_+$.  
\end{rem}

If one is interested in establishing only the equivalence of statements (i) and (iii) of Theorem~\ref{Thm}, there is a shorter argument which may be of independent interest. Recall the classical fact that, given a representation $T$ of a group $G$ by isometries on a Banach space $X$, one has $\widehat{a}(T)=0$ for all $a\in L^1(G)$ which are of spectral synthesis with respect to $\Sp(T)$. This is a simple consequence of the definition of the spectrum (see also \cite[Chapter~8]{Dav},  \cite[Chapter~5]{Lyu} and \cite[Lemma~2.4.3]{vNBirk}) and is used (together with constructions analogous to those used in the proof of Proposition~\ref{pwlim}) in \cite[Theorem~4.3]{BV} to derive a general form of the Katznelson-Tzafriri theorem on Banach space. Corollary~\ref{synth0} below, which is an improved version of this result when $X$ is a Hilbert space, makes it possible to obtain the implication (i)$\implies$(iii) of  Theorem~\ref{Thm} by an analogous argument which bypasses Proposition~\ref{zero}. It is a special case of the following more general statement.

\begin{myprp}\label{synth}
Let $T$ be a representation of a group $G$ by unitary operators on a Hilbert space $X$, and let $a\in L^1(G)$. Then $\|\widehat{a}(T)\|=\sup\{|\widehat{a}(\chi)|:\chi\in\Sp(T)\}.$
\end{myprp}

\begin{proof}[\textsc{Proof}]Let $\mathcal{A}_T$  denote the norm closure in $\mathcal{B}(X)$ of $\{\widehat{b}(T):g\in L^1(G)\}$. Then $\mathcal{A}_T$ is a commutative $C^*$-algebra and hence, writing $\Delta(\mathcal{A}_T)$ for the character space of $\mathcal{A}_T$, the Gelfand transform $\Phi:\mathcal{A}_T\rightarrow C_0(\Delta(\mathcal{A}_T))$ is an isometric $*$-isomorphism. By \cite[Proposition~2.4]{BV}, the map sending a character $\chi\in\Sp(T)$ to the character $\xi_\chi$ on $\mathcal{A}_T$ defined, on the dense subspace $\{\widehat{b}(T):g\in L^1(G)\}$, by $\xi_\chi(\widehat{b}(T)):=\widehat{b}(\chi)$ is a bijection, and hence $\|\widehat{a}(T)\|=\|\Phi(\widehat{a}(T))\|_\infty=\sup\{|\widehat{a}(\chi)\big|:\chi\in\Sp(T)\}$.
\end{proof}

\begin{rem}
This result  can also be proved using \eqref{spectral} with $\Lambda=\Sp(T)$.
 \end{rem}

\begin{mycor}\label{synth0}
Let $T$ be a representation of a group $G$ by unitary operators on a Hilbert space $X$ with spectrum $\Lambda:=\Sp(T)$, and let $a\in L^1(G)$. Then  $\widehat{a}(T)=0$  if and only if $a\in K_\Lambda$.
\end{mycor}

\begin{rem}
This follows also from Corollary~\ref{cor:erg_cond} and Proposition~\ref{zero}  together with Lemma~\ref{converse2}. 
\end{rem}

\section{Quantified results for contractive representations}

The purpose of this section is to study the limit of $\|T(s)\widehat{a}(T)\|$ as $s\rightarrow\infty$ in the case where $T$ is a contractive representation of a semigroup $S$ on a Hilbert space $X$ and $a$ is an element of $L^1(S)$ whose Fourier transform $\widehat{a}$ does not necessarily vanish on the unitary spectrum $\Spu(T)$ of $T$.  Theorem~\ref{ub} below constitutes an important step towards this aim and can be viewed as a sharper version of \cite[Proposition~5.5]{BBG} which holds on general Banach space. It  follows from the following result for individual orbits.

\begin{myprp}\label{pwub}
Let $T$ be a representation of  a semigroup $S$ by contractions on a Hilbert space $X$ with unitary spectrum $\Lambda:=\Spu(T)$, and let $a\in L^1(S)$. Then$$\lim_{s\rightarrow\infty}\|T(s)\widehat{a}(T)x\|\leq \|a+K_\Lambda\|\|x\|$$for all $x\in X$.
\end{myprp}

\begin{proof}[\textsc{Proof}]
Fix any Banach limit $\phi$ on $L^\infty(S)$ and let $X_\phi$, $T_\phi$ and $\pi_\phi$ be as in the proof of Proposition~\ref{pwlim}. Then, since $T$ is contractive, $\|\pi_\phi(x)\|=\lim_{s\rightarrow\infty}\|T(s)x\|$ for all $x\in X$ and it is possible, as before, to assume that $T_\phi$ is in fact a representation of the group $G=S-S$ on $X_\phi$ by unitary operators. It follows from Corollary~\ref{synth0} that $\widehat{b}(T_\phi)=0$ for all $b\in K_\Lambda$. Hence $
\|\widehat{a}(T_\phi)\|\leq \|a-b\|_1$
for any such $b$, which implies that $\|\widehat{a}(T_\phi)\|\leq \|a+K_\Lambda(G)\|$. Thus, given any $x\in X$, 
$$\begin{aligned}
\lim_{s\rightarrow\infty}\|T(s)\widehat{a}(T)x\|&=\|\pi_\phi(\widehat{a}(T)x)\|\\&=\|\widehat{a}(T_\phi)\pi_\phi(x)\|\\&\leq \|a+K_\Lambda(G)\| \|\pi_\phi(x)\|,
\end{aligned}$$and the result follows since $\pi_\phi$ is a contraction.
\end{proof}

\begin{mythm}\label{ub}
Let  $T$ be a representation of $S$ by contractions on a Hilbert space $X$ with unitary spectrum $\Lambda:=\Spu(T)$, and let $a\in L^1(S)$. Then $$\lim_{s\rightarrow\infty}\|T(s)\widehat{a}(T)\|\leq \|a+K_\Lambda\|.$$ 
\end{mythm}

\begin{proof}[\textsc{Proof}]
Suppose not. Then there exist $\varepsilon>0$ and a net $(s_\beta)$ in $S$, with indexing set $A$, say, and satisfying $s_\beta\rightarrow\infty$ as $\beta\rightarrow\infty$, as well as a net of unit vectors $(y_\beta)$ in $X$ such that $\|T(s_\beta)\widehat{a}(T)y_\beta\|\geq \|a+K_\Lambda\|+\varepsilon$ for all $\beta\in A$. In particular, $\|T(s)\widehat{a}(T)y_\beta\|\geq \|a+K_\Lambda\|+\varepsilon$ whenever $s_\beta-s\in S$.

 Let $X_A$,  $X_A^\infty$ , $T_A$ and  $\pi_A$ be as in the proof of  Proposition~\ref{lim}, choose, for a fixed  $t\in S^\circ$,  $b\in L^1(S)$  such that $\|b\|_1=1$ and $\|a*b-a_{t}\|_1<\varepsilon/2$, and again define $(x_\beta)\in X_A^\infty$ by setting $x_\beta:=\widehat{b}(T)y_\beta$. It then follows from a calculation analogous to the one in the proof of Proposition~\ref{lim} that $\|T_A(s)\widehat{a}(T_A)\pi_A(x_\beta)\| > \|a+K_\Lambda\|+\varepsilon/2$ for all $s\in S$. However, applying  Proposition~\ref{pwub} to the contractive representation $T_A$ of $S$ on $X_A$, and using the fact that $\|\pi_A(x_\beta)\|\leq \|\widehat{b}(T)\|\leq1,$ leads to $$\begin{aligned}\lim_{s\rightarrow\infty}\|T_A(s)\widehat{a}(T_A)\pi_A(x_\beta)\|&\leq  \|a+K_{\Lambda_A}\| \|\pi_A(x_\beta)\|\leq \|a+K_{\Lambda}\|,\end{aligned}$$ where $\Lambda_A:=\Spu(T_A)\subset \Lambda$. This gives the required contradiction.
\end{proof}

The final  result is a special instance Theorem~\ref{ub} which applies to  representations whose unitary spectrum is  a Helson set. It is a simple consequence of the definition of a Helson set and should be compared with the results in \cite[Section~5]{Zarr}, which hold on Banach space but in addition assume the unitary spectrum to be of spectral synthesis. 

\begin{mycor}\label{Helub}
Let $T$ be a representation of a semigroup $S$ by contractions on a Hilbert space $X$, and let  $a\in L^1(S)$. Suppose that the unitary spectrum  $\Lambda:=\Spu(T)$ of $T$  is a Helson set. Then \begin{equation*}\label{UB} \sup\{|\widehat{a}(\chi)|:\chi\in \Lambda\}\leq \lim_{s\rightarrow\infty}\|T(s)\widehat{a}(T)\|\leq \alpha(\Lambda) \sup\{|\widehat{a}(\chi)|:\chi\in \Lambda\}.\end{equation*} 
\end{mycor}

\begin{rem}
The first inequality holds irrespective of whether the unitary spectrum is a Helson set, and indeed of whether $X$ is a Hilbert space. It is an immediate consequence of Lemma~\ref{converse2}.
\end{rem}

\section{Local results}

This final section gives a brief account of how some of the aforementioned improvements of the Katznelson-Tzafriri theorem on Hilbert space carry over to orbitwise, or `local', versions of the result. More specifically, the aim  is to obtain spectral conditions which ensure, given a bounded representation $T$ of a semigroup on a Hilbert space $X$ and an element $a$ of $L^1(S)$, that $\|T(s)\widehat{a}(T)x\|\rightarrow0$ as $s\rightarrow\infty$ for a \textsl{particular} point $x\in X$. Such results are of particular interest in the context of $C_0$-semigroups, where orbits correspond to solutions of the associated Cauchy problem, and they have been studied for instance in \cite{BvNR}, \cite[Section~4]{BY2} and \cite{KvN}.

The  notion of \textsl{spectrum} that is most appropriate to this context goes back to \cite{Alb}. Consider a bounded representation $T$ of a semigroup $S$ on some Banach space $X$, and let $x\in X$ be given.  A character $\chi\in \Gamma$ will be said to be \textsl{locally regular at $x$} if there exist $n\in\mathbb{N}$, $a_1,\dotsc,a_n\in L^1(S)$, a neighbourhood $\Omega$ of the point $\lambda_0:=(\widehat{a_1}(\chi),\dotsc,\widehat{a_n}(\chi))$ in $\mathbb{C}^n$ and holomorphic functions $g_1,\dotsc,g_n:\Omega\rightarrow X$ such that $\sum_{k=0}^n(\lambda_k-\widehat{a_k}(T))g_k(\lambda)=x$ for all $\lambda=(\lambda_1,\dotsc,\lambda_n)\in\Omega$. The \textsl{unitary local (Albrecht) spectrum} $\Spu(T;x)$ of $T$ at $x$ is then defined  as the set of all characters $\chi\in \Gamma$ which fail to be locally regular at $x$. It is easy to see that $\Spu(T;x)\subset\Spu(T)$ for each $x\in X$. For further details on the unitary local spectrum and its relation to other natural notions of local spectrum, see \cite[Section~4]{BY2}.

The main result of this section is the following theorem, which improves  \cite[Theorem~5.1]{BvNR} in the Hilbert space setting.

\begin{mythm}\label{local_KT}
Let $T$ be a bounded representation of a semigroup $S$ on a Hilbert space $X$. Furthermore, let $x\in X$ and $a\in L^1(S)$. Then $\|T(s)\widehat{a}(T)x\|\rightarrow0$ as $s\rightarrow\infty$ provided $\widehat{a}(\chi)=0$ for all $\chi\in\Spu(T;x)$. 
\end{mythm}

\begin{proof}[\textsc{Proof}] 
Fix a Banach limit $\phi$ on $L^\infty(S)$ and let  $X_\phi$, $T_\phi$ and $\pi_\phi$ be as in the proof of Proposition~\ref{pwlim}, so that $T_\phi$ may again be assumed to be a representation of the group $G=S-S$ on $X_\phi$ by unitary operators. By \cite[Proposition~5.1]{BY2}, $\Spu(T_\phi;\pi_\phi(x))\subset\Spu(T;x)$. Moreover, writing $X_x$ for the closed linear span of the set $\{T_\phi(s)\pi_\phi(x):s\in G\}$ in $X_\phi$ and $T_x$ for the representation of $G$ on $X_x$ obtained by restricting $T_\phi$, \cite[Theorem~4.5]{BY2} gives $\Sp(T_x)=\Spu(T_\phi;\pi_\phi(x))$ and hence $\Sp(T_x)\subset\Spu(T;x)$.  Thus the assumption on $a$ implies that $\widehat{a}(\chi)=0$ for all $\chi\in \Sp(T_x)$, from which it follows by Corollary~\ref{synth0} that $\widehat{a}(T_x)=0$. Hence $\pi_\phi(\widehat{a}(T)x)=\widehat{a}(T_\phi)\pi_\phi(x)=0$, which is to say  $\widehat{a}(T)x\in\Ker\pi_\phi$, as required.
\end{proof}

Theorem~\ref{local_KT}  in fact holds even for certain unbounded representations provided the growth of the norm is sufficiently slow and regular. Given a semigroup $S$, a measurable  function $w:S\rightarrow[1,\infty)$ is said to be a \textsl{weight} if it is  bounded on compact subsets of  $S$ and satisfies $w(s+t)\leq w(s)w(t)$ for all $s,t\in S$. Given a  weight $w$ and a representation $T$ of $S$ on a Banach space $X$ which satisfies $\|T(s)\|\leq w(s)$ for all $s\in S$, it is possible, essentially by replacing any occurrence of $L^1(S)$ with the Beurling algebra $L^1_w(S)$, to define the modified  unitary local (Albrecht) spectrum $\Spuw(T;x)$ for any point $x\in X$; see  \cite{BY} and \cite{Lyu} for details.

An argument entirely analogous to the proof of Theorem~\ref{local_KT}, but this time using the full strength of the results in \cite{BY2}, then leads to the following result. Here the additional regularity assumption  on the weight $w$ ensures that the representation corresponding to $T_\phi$ in the above proof is again isometric; see \cite[Proposition~3.1]{BY2}. Similar non-local results may be found in \cite[Theorem~3.4]{BY2} and \cite[Theorem~8]{Vu2}.

\begin{mythm}\label{local_nqa_KT}
Let $T$ be a representation of a semigroup $S$ on a Hilbert space $X$ and suppose that $T$ is dominated by a weight $w$ such that, for every $t\in S$,  $w(s)^{-1}w(s+t)\rightarrow1$ as $s\rightarrow\infty$. Furthermore, let $x\in X$ and suppose that $a\in L_w^1(S)$ is such that $\widehat{a}(\chi)=0$ for all $\chi\in \Spuw(T;x)$. Then $\|T(s)\widehat{a}(T)x\|=o(w(s))$ as $s\rightarrow\infty$.
\end{mythm}

\section*{Acknowledgements}

The author is grateful to Professor C.J.K.\ Batty for his guidance, to the EPSRC for its financial support, and to the anonymous referee, whose careful reading of an earlier version  led to various minor improvements.


\begin{thebibliography}{10}

\bibitem{Alb}
E.~Albrecht.
\newblock Spectral decompositions for systems of commuting operators.
\newblock {\em Math. Proc. R. Ir. Acad. (Section A)}, 81:81--98, 1981.

\bibitem{AR}
G.R. Allan and T.J. Ransford.
\newblock {Power-dominated elements in a Banach algebra}.
\newblock {\em Studia Math.}, 94:63--79, 1989.

\bibitem{BBG}
C.J.K. Batty, Z.~Brze\'zniak, and D.A. Greenfield.
\newblock A quantative asymptotic theorem for contraction semigroups with
  countable unitary spectrum.
\newblock {\em Studia Math.}, 121(2):167--183, 1996.

\bibitem{BG}
C.J.K. Batty and D.A. Greenfield.
\newblock On the invertibility of isometric semigroup representations.
\newblock {\em Studia Math.}, 110(3):235--250, 1994.

\bibitem{BvNR}
C.J.K. Batty, J.M.A.M. van Neerven, and F.~R\"abiger.
\newblock {Local spectra and individual stability of uniformly bounded
  $C_0$-semigroups}.
\newblock {\em Trans. Amer. Math. Soc.}, 350(5):2071--2085, 1998.

\bibitem{BV}
C.J.K. Batty and Q.P. V$\tilde{\mbox{u}}$.
\newblock {Stability of strongly continuous representations of abelian
  semigroups}.
\newblock {\em Math. Z.}, 209:75--88, 1992.

\bibitem{BY2}
C.J.K. Batty and S.B. Yeates.
\newblock Weighted and local stability of semigroups of operators.
\newblock {\em Math. Proc. Cambridge Philos. Soc.}, 129:85--98, 2000.

\bibitem{BY}
C.J.K. Batty and S.B. Yeates.
\newblock Extensions of semigroups of operators.
\newblock {\em J. Operator Theory}, 46:139--157, 2001.

\bibitem{Berc2}
H.~Bercovici.
\newblock {The asymptotic behaviour of bounded semigroups in a Banach algebra}.
\newblock {\em Math. Proc. Cambridge Philos. Soc.}, 108(2):365--370, 1990.

\bibitem{Berc}
H.~Bercovici.
\newblock Commuting power-bounded operators.
\newblock {\em Acta Sci. Math. (Szeged)}, 57(1-4):55--64, 1993.

\bibitem{CT}
R.~Chill and Y.~Tomilov.
\newblock Stability of operator semigroups: ideas and results.
\newblock In {\em Perspectives in Operator Theory}. Banach Center Publications,
  Volume 75, Polish Academy of Sciences, Warsaw, 2007.

\bibitem{Dav}
E.B. Davies.
\newblock {\em One-Parameter Semigroups}.
\newblock Academic Press, London, 1980.

\bibitem{Dou}
R.G. Douglas.
\newblock On extending commutative semigroups of isometries.
\newblock {\em Bull. Lond. Math. Soc.}, 1:157--159, 1969.

\bibitem{ESZ2}
J.~Esterle, E.~Strouse, and F.~Zouakia.
\newblock {Theorems of Katznelson-Tzafriri type for contractions}.
\newblock {\em J. Funct. Anal.}, 94:273--287, 1990.

\bibitem{ESZ}
J.~Esterle, E.~Strouse, and F.~Zouakia.
\newblock Stabilit\'e asymptotique de certaines semi-groupes d'op\'erateurs et
  ideaux primaires de {$L^1(\mathbb{R}^+)$}.
\newblock {\em J. Operator Theory}, 28:203--227, 1992.

\bibitem{HR}
S.-Z. Huang and F.~R\"abiger.
\newblock {Superstable $C_0$-semigroups on Banach spaces}.
\newblock In P.~Cl\'ement and G.~Lumer, editors, {\em Evolution Equations,
  Control Theory and Biomathematics}, volume 155 of {\em Lecture Notes in Pure
  and Applied Mathematics}, pages 291--300. Dekker, New York, 1994.

\bibitem{Ito}
T.~Ito.
\newblock On the commutative family of subnormal operators.
\newblock {\em Journal of the Faculty of Science, Hokkaido University (Series
  1)}, 14:1--15, 1958.

\bibitem{KT}
Y.~Katznelson and L.~Tzafriri.
\newblock On power bounded operators.
\newblock {\em J. Funct. Anal.}, 68:313--328, 1986.

\bibitem{KvN}
L.~K{\'e}rchy and J.M.A.M. van Neerven.
\newblock Polynomially bounded operators whose spectrum on the unit circle has
  measure zero.
\newblock {\em Acta Sci. Math. (Szeged)}, 63:551--562, 1997.

\bibitem{Leka}
Z.~L\'eka.
\newblock {A Katznelson-Tzafriri type theorem in Hilbert spaces}.
\newblock {\em Proc. Amer. Math. Soc.}, 137(11):3763--3768, 2009.

\bibitem{Lyu}
Y.I. Lyubich.
\newblock {\em Introduction to the Theory of Banach Representations of Groups}.
\newblock Birkh\"auser, Basel, 1988.

\bibitem{Mus}
H.S. Mustafayev.
\newblock Asymptotic behaviour of polynomially bounded operators.
\newblock {\em C. R. Math. Acad. Sci. Paris}, 348:517--520, 2010.

\bibitem{NR}
R.~Nagel and F.~R\"abiger.
\newblock Superstable operators on {Banach spaces}.
\newblock {\em Israel J. Math.}, 81:213--226, 1993.

\bibitem{Ped}
G.K. Pedersen.
\newblock {\em C*-Algebras and their Automorphism Groups}.
\newblock Academic Press, San Diego, 1979.

\bibitem{RW}
F.~R\"abiger and M.P.H. Wolff.
\newblock Superstable semigroups of operators.
\newblock {\em Indag. Math. (N.S.)}, 6(4):481--494, 1995.

\bibitem{Rudin}
W.~Rudin.
\newblock {\em Fourier Analysis on Groups}.
\newblock Wiley, New York, 1962.

\bibitem{vNBirk}
J.M.A.M. van Neerven.
\newblock {\em The Asymptotic Behaviour of Semigroups of Linear Operators}.
\newblock Birkh\"auser, Basel, 1996.

\bibitem{Vu}
Q.P. V$\tilde{\mathrm{u}}$.
\newblock {Theorems of Katznelson-Tzafriri type for semigroups of operators}.
\newblock {\em J. Funct. Anal.}, 103:74--84, 1992.

\bibitem{Vu2}
Q.P. V$\tilde{\mathrm{u}}$.
\newblock Semigroups with nonquasianalytic growth.
\newblock {\em Studia Math.}, 103:229--241, 1993.

\bibitem{Zarr}
M.~Zarrabi.
\newblock {Some results of Katznelson-Tzafriri type}.
\newblock {\em J. Math. Anal. Appl.}, 397:109--118, 2013.

\end{thebibliography}
\end{document}